\documentclass[a4paper,11pt,twoside]{amsart}

\usepackage[latin1]{inputenc}
\usepackage{amsmath, amsfonts, amssymb,amsthm}
\usepackage[mathscr]{eucal} 
\usepackage[pdftex]{graphicx}
\usepackage{color}

\usepackage[T1]{fontenc}
\usepackage[sc]{mathpazo}
\usepackage[all]{xy}
\usepackage{hyperref}
\usepackage[flenq]{nccmath}

\usepackage {anysize}
\marginsize{2cm}{2cm}{2cm}{3cm}

\usepackage{enumerate}

\definecolor{alert}{rgb}{0.8,0,0}

\newcommand{\R}{\mathbb{R}}

\newcommand{\s}{\mathbb{S}}
\newcommand{\h}{\mathbb{H}}

\newcommand{\Nil}{\mathrm{Nil}_3(\tau)}
\newcommand{\Sol}{\mathrm{Sol}_3}
\newcommand{\fl}{\longrightarrow}
\newcommand{\oz}{\overline{z}}

 \newcommand{\p}{\widetilde{PSL}_2(\mathbb{R},\tau)}
 
\newcommand{\real}{\mathbb{R}}

\DeclareMathOperator{\sech}{sech}

\newtheorem{theorem}{Theorem}[section]
\newtheorem{proposition}[theorem]{Proposition}
\newtheorem{corollary}[theorem]{Corollary}
\newtheorem{lemma}[theorem]{Lemma}

\theoremstyle{definition}

\theoremstyle{remark}
\newtheorem{remark}[theorem]{Remark}

\numberwithin{equation}{section}

\title[Extrinsic curvature]{The Gauss map and second fundamental form of
surfaces\\ in a Lie group}

\author{Abigail Folha
and Carlos Pe\~{n}afiel}
\address{}
\email{}

\thanks{This work was partially supported by CNPq, Conselho Nacional de Desenvolvimento Cient{\'i}fico e Tecnol{\'o}gico - Brazil.}



\begin{document}

\begin{abstract}
In this article, we give the integrability conditions for the existence of an isometric immersion from an orientable simply connected surface having prescribed Gauss map and  positive extrinsic curvature into some unimodular Lie groups. In particular, we discuss the case when the Lie group is the euclidean unit sphere $\s^3$ and establish a correspondence between simply connected surfaces having extrinsic curvature $K$, $K$ different from $0$ and $-1$,  immersed in  $\s^3$ with simply connected surfaces having non-vanishing extrinsic curvature immersed in the euclidean space $\R^3$. Moreover, we show that a surface isometrically immersed in $\s^3$ has positive constant extrinsic  curvature  if, and only if, its Gauss map is a harmonic map into the Riemann sphere.
 \end{abstract}

\maketitle

\section{Introduction}
In the classical theory of surfaces, the existence and uniqueness of a surface with prescribed Gauss map and for which the first fundamental form is a conformal metric were deeply studied. For instance, in the euclidean case $\R^3$, Kenmotsu \cite{Ke} gave a remarkable representation formula for arbitrary surfaces with non-vanishing mean curvature, which describes these surfaces in terms of their Gauss map and mean curvature function. For the hyperbolic space $\h^3$, the study of such surfaces and their geometry was done in \cite{Reiko} and for the euclidean unit sphere  $\s^3$, this topic was treated in \cite{Tanaka}. Recently for the homogeneous 3-manifolds, this topic was treated in \cite{GP}. In particular, those surfaces having constant mean curvature  were object of a huge amount of investigation.

It is well known that surfaces immersed in a space form having non-zero constant mean curvature are parallel to surfaces having positive Gaussian curvature (see, for instance  \cite[Section 3]{Tenenblat}). Under this framework the natural question which arises is about the description of surfaces in terms of their Gauss map and extrinsic  curvature function. In this line, for the euclidean case $\R^3$, the problem of existence of an immersion from a Riemann surface $S$ having a given Gauss map and a prescribed extrinsic curvature  was treated in \cite{GM}.  The study of this problem for the hyperbolic case $\h^3$ was treated in \cite{Shi}. Finally,  in \cite{BIK} was treated the case of immersions in the  euclidean sphere $\s^3$ having  negative extrinsic curvature $K, -1<K<0$. In all these works the key is to find a partial differential equation involving the Gauss map and the extrinsic curvature which is equivalent to the integrability conditions.   

In this article, in order to define a Gauss map we work on Lie groups, and we  focus our attention on unimodular Lie groups. More precisely, we are interested in the following unimodular Lie groups, 
the euclidean space $\R^3$, the canonical sphere $\s^3$, the Berger spheres, the space $\p$, the Heisenberg space $\Nil$, the space $\Sol$, see Section \ref{unimodular}.

Using the group structure, we are able to find a partial differential equation involving the Gauss map of the immersion and the extrinsic  curvature  which is equivalent to the integrability equations for positive extrinsic curvature, see Theorem \ref{t1}. When the extrinsic curvature $K$ is negative,  there is an additional  restriction on the values of $K$, see  Theorem \ref{t3}.  In particular, we prove that a surface isometrically immersed in the euclidean sphere $\s^3$ has constant  positive extrinsic curvature $K$ if, and only if, the Gauss map $g$ is a harmonic map into the Riemann sphere $\s^2$.

In the  case when the unimodular Lie group $M$ is the euclidean sphere $\s^3$, we also prove that, under some hypothesis,  there exists a correspondence between  simply connected surfaces immersed in $\s^3$ having positive (negative, different from -1) extrinsic curvature  $K$ with simply connected surfaces immersed in $\real^3$ having positive (negative)  extrinsic curvature  $K^*$, see Proposition \ref{pp1} for $K$ positive and Proposition \ref{p5} for a negative $K$.  We apply these results to construct an immersion in $\s^3$ having constant extrinsic curvature $K=-2$.

The outline of the article goes as follows. In Section 2 we deal with surfaces immersed in a Riemannian manifold having positive extrinsic curvature. Taking a conformal parameter for the second fundamental form, we establish a relationship which enable us to recover the immersion from the normal vector, the extrinsic curvature and the basis point, see Proposition \ref{p1}. In Section 3, we define the Gauss map and we present the Lie groups which we are interested in. We show the PDE involving the Gauss map and the extrinsic curvature which is satisfied by an arbitrary immersed surface. Also we give the integrability conditions for the existence of such immersion, see Theorem \ref{t1}. We finish section 3 considering the case where the Lie group is the euclidean sphere $\s^3$. In Section 4, we consider orientable isometric immersion having negative extrinsic curvature. Finally in Section 5, we construct an example of a surface having constant extrinsic curvature $K=-2$.

\section{Surfaces with positive extrinsic curvature}

Let $\Sigma$ be a Riemannian surface and $\varphi : \Sigma\fl M$ be an isometric immersion in an oriented 3-dimensional Riemannian manifold $M$. Assume that $\varphi$ has positive extrinsic curvature $K$, that is, such that the determinant of its shape operator is positive. Then, $\Sigma$ must be orientable and we can choose a global unit normal vector field $N$ such that the second fundamental form $II$ is positive definite. Given a local conformal parameter $z$ for the Riemannian metric $II$, the first and second fundamental forms of the immersion can be written as
$$
I:=\langle d \varphi, d \varphi\rangle = E\, dz^2+ 2F\, \vert dz\vert^2+ \overline{E}\,d\overline{z}^2,
$$
$$
II:= \langle - d N, d\varphi \rangle=\  2\rho\,  \vert dz\vert^2,
$$
where $\rho$ is a positive function and  $\overline{z}$ denotes the conjugate of $z$. The extrinsic curvature of $\Sigma$ is given by $K=-\dfrac{\rho^2}{D}$, where $D= \vert E\vert^2- F^2<0$.

Thus, the Weingarten equation remains
\begin{eqnarray}\label{e4}
\nabla_{\partial_z} N&=&\dfrac{\rho}{D}\,(F\, \partial_z -E\, \partial_{\overline{z}}),
\end{eqnarray}
where, for instance, $\partial_z$ denotes $\frac{\partial\ }{\partial z}$.

Since $z$ is a conformal parameter for the second fundamental form,  $\partial_z$ is simultaneously orthogonal to $N$ and to $\nabla_{\partial_z} N$, and so
$$
\partial_z=\alpha\  N\times_\varphi \nabla_{\partial_z} N,
$$
for some complex function $\alpha$, where $\times_\varphi$ denotes the cross product at the tangent space $T M$ of $M$ at the point $\varphi(z)$. Considering the inner product of $\partial_z$ with  $\nabla_{\partial_{\overline{z}}} N$ we obtain
\[
-\rho= \langle \partial_z,  \nabla_{\partial_{\overline{z}}} N\rangle
=\alpha\langle N\times_\varphi \nabla_{\partial_z} N, \nabla_{\partial_{\overline{z}}} N\rangle
=\alpha\,K\langle N,\partial_z\times_{\varphi}\partial_{\overline{z}}\rangle
= \alpha\,K\vert  \partial_z\times_\varphi\partial_{\overline{z}}\vert=i\,\alpha\,K\,\sqrt{-D},
\]
that is
\begin{equation*}
\partial_z=\dfrac{i}{\sqrt{K}}\  N\times_\varphi \nabla_{\partial_z} N.
\end{equation*}

Thus, the integrability equation of the immersion, $0=[\partial_z,\partial_{\oz}]$, is given by
\begin{eqnarray*}
0&=&\nabla_{\partial_z} \left(\dfrac{-i}{\sqrt{K}}\  N\times_\varphi \nabla_{\partial_{\oz}} N\right)-\nabla_{\partial_{\oz}} \left(\dfrac{i}{\sqrt{K}}\  N\times_\varphi \nabla_{\partial_z} N\right)\\
&=&(-i)\ N\times_{\varphi}\left(\nabla_{\partial_z}\left(\frac{1}{\sqrt{K}}\nabla_{\partial_{\oz}} N\right)+\nabla_{\partial_{\oz}}\left(\frac{1}{\sqrt{K}}\nabla_{\partial_{z}} N\right)\right).
\end{eqnarray*}

Summarizing, we have
\begin{proposition}\label{p1}
Let $\varphi:\Sigma\fl M$ be an isometric immersion from a Riemannian surface $\Sigma$ into a 3-dimensional Riemannian manifold $M$. Assume that $\varphi$ has positive extrinsic curvature $K$, and let $N$ be a unit normal vector field. Then,
\begin{equation}\label{e6}
\partial_z=\dfrac{i}{\sqrt{K}}\  N\times_\varphi \nabla_{\partial_z} N,
\end{equation}
where $z$ is a local conformal parameter for the second fundamental form.

Moreover, the integrability equation, $0=[\partial_z,\partial_{\oz}]$, is equivalent to
\begin{equation}
\label{e8}
N\times_{\varphi}\left(\nabla_{\partial_z}\left(\frac{1}{\sqrt{K}}\nabla_{\partial_{\oz}} N\right)+\nabla_{\partial_{\oz}}\left(\frac{1}{\sqrt{K}}\nabla_{\partial_{z}} N\right)\right)=0.
\end{equation}
\end{proposition}

Observe that when $M=\R^3$, the cross product does not depend on the point $\varphi(z)$ and so the immersion can be recovered, from (\ref{e6}), as
$$
\varphi_z=\dfrac{i}{\sqrt{K}}\  N\times N_z,
$$
(see also \cite{GHM,GM}). Moreover, if $K$ is a positive constant then the integrability equation is equivalent to $N\times N_{z\oz}=0$, that is, $N:\Sigma\fl\s^2$ is a harmonic map for the conformal structure induced by the second fundamental form.

For a general ambient space $M$ the cross product depend on the point $\varphi(z)$ and so the same process can not be followed as in $\R^3$. However, when $M$ is a Riemannian Lie group an alternative procedure can be used in order to recover the immersion $\varphi$ in terms of its Gauss map and extrinsic curvature, using (\ref{e6}).

\section{The Gauss map}

From now on, we assume that $M$ is a Riemannian Lie group with a left invariant metric, and let us choose a left invariant orthonormal frame $\{E_1(q), E_2(q), E_3(q)\},$ $q\in M$.  For simplicity $\{E_1, E_2, E_3\}$  stands for $\{E_1(e), E_2(e), E_3(e)\}$, where $e$ is the identity element of $M$; that is, $E_i(q)=q\, E_i$, where we are using the product of the group $M$. A vector $w=w_1 E_1+w_2 E_2+ w_3 E_3 $ in the Lie algebra $\mathfrak{m}$ of $M$ will be denoted by $w\equiv(w_1, w_2, w_3)$.

Let $\varphi:\Sigma\fl M$ be an isometric immersion with positive extrinsic curvature, and $N$ be the normal vector to $\Sigma$ such that the second fundamental form is positive definite.

Given a point $p\in \Sigma$, we extend the normal vector $N(p)\in T_p M,$ to a unique left invariant vector field. Then, $N(p)$ is associated to a vector $N^e(p)$ in the Lie algebra $\mathfrak{m}$ of $M$ given by
\[
N^e(p):= \sum_{i=1}^{3} \langle N(p), E_i(\varphi(p))\rangle E_i,
\]
or equivalently, using the group structure of $M$, $N^e(p)=\varphi(p)^{-1}\,N(p)$.

$N^e:\Sigma\fl\s^2\subseteq\mathfrak{m}$ is called the \emph{Gauss map} of the immersion. As usual, we will also call Gauss map to the composition $g: \Sigma\fl \mathbb{C}\cup \{\infty\}$ of the stereographic projection  (with respect to the north pole) and $N^e$, that is,
 \begin{equation}
\label{e2}
g= \dfrac{N_1+ i\,  N_2}{1-N_3}.
\end{equation}
Conversely, $N^e$ can be recovered from $g$ as
\begin{equation}
\label{e3}
N^e\equiv(N_1,N_2,N_3)= \dfrac{1}{1+\vert g\vert^2}\,(g+\overline{g},\  -i\,(g-\overline{g}),\  -1+\vert g\vert^2).
\end{equation}

Now, let $z$ be a local conformal parameter for the second fundamental form, and write
$$
\varphi(p)^{-1}\,\varphi_z(p)=\sum_{i=1}^{3} \langle \varphi_z(p), E_i(\varphi(p))\rangle E_i=:a_1 E_1+ a_2 E_2+ a_3E_3,
$$
where $\varphi_z$ denotes $\partial_z$ in the Lie group $M$.

Since $M$ is a Riemannian Lie group, we can rewrite (\ref{e6}) as
\begin{equation}
\label{e11}
(a_1, a_2, a_3)=\dfrac{i}{\sqrt{K}} (N_1, N_2, N_3) \times \left((N_1)_z+\sum_{i,j} N_i a_j \Gamma_{ij}^1\, , \,(N_2)_z+\sum_{i,j} N_i a_j \Gamma_{ij}^2\, , \, (N_3)_z+\sum_{i,j} N_i a_j \Gamma_{ij}^3 \right),
\end{equation}
where $\Gamma_{ij}^k$ are the Christoffel symbols associated with the orthonormal basis, that is, $\nabla_{E_i(q)}E_j(q)=\sum_{k=1}^3\Gamma_{ij}^kE_k(q)$. Observe that $\Gamma_{ij}^k$ are real constants since $\{E_1(q),E_2(q),E_3(q)\}$ is a left invariant basis. Here, $\times$ denotes the standard cross product in the Lie algebra $\mathfrak{m}$.

So, (\ref{e11}) can be considered as a system of linear equations in the unknowns $a_1,a_2,a_3$. Hence, as long as the discriminant of the above system is different from zero, we can determine $\varphi(p)^{-1}\,\varphi_z(p)$ in terms of the Gauss map and the extrinsic curvature of the immersion. As we will see, once $\varphi(p)^{-1}\,\varphi_z(p)$ is calculated we can compute the immersion $\varphi$, up to left translations. Thus, we will say that the immersion is determined in terms of its Gauss map and its extrinsic curvature.

Moreover, note that $N$ and $K$ must satisfy the integrability condition (\ref{e8}), which can again be rewritten in terms of $N^e$ and $K$ using left translations in the Riemannian Lie group $M$. That is, if we denote by
$
V\equiv(v_1,v_2,v_3)$ the left translation of $\nabla_{\varphi_z}\left(\frac{1}{\sqrt{K}}\nabla_{\varphi_{\oz}} N\right)
$ to the identity element $e$, then (\ref{e8}) is equivalent to
\begin{eqnarray}\label{e8bis}
&N^e \text{ is parallel to  }\mathfrak{Re}(V), \text{ with} \\
&
v_k=\left(\frac{1}{\sqrt{K}}\left((N_k)_{\oz}+\sum_{i,j}\overline{a_i}N_j\Gamma_{ij}^k\right)\right)_{\! z}+\frac{1}{\sqrt{K}}\sum_{i,j,l,m}\left((N_m)_{\oz}+\overline{a_i}N_j\Gamma_{ij}^m\right)a_l\Gamma_{lm}^k,\nonumber
\end{eqnarray}
where $\mathfrak{Re}(\cdot)$ denotes the real part of a complex number.

Therefore, the Gauss map and extrinsic curvature of the immersion $\varphi$ must satisfy (\ref{e11}) and (\ref{e8bis}). Conversely, we obtain

\begin{theorem}\label{teo1}
Let $M$ be a Riemannian Lie group, and $\{E_1(q),E_2(q),E_3(q)\}$ a left invariant orthonormal frame with associated Christoffel symbols $\Gamma_{ij}^k$. Consider a simply connected Riemann surface $\Sigma$, and let $N^e:\Sigma\fl\s^2\subseteq\mathfrak{m}$ be a map, and $K:\Sigma\fl\R$ be a positive function. Assume that $(a_1,a_2,a_3)$ is a smooth solution to (\ref{e11}), so that (\ref{e8bis}) is satisfied.

Then, there exists a unique immersion $\varphi:\Sigma\fl M$, up to left translations, such that $N^e$ is its Gauss map, $K$ is its extrinsic curvature and the conformal structure of $\Sigma$ is that of the second fundamental form induced by $\varphi$, with
\begin{equation}\label{ecuacion2}
\varphi_z(p)=a_1(p)E_1(\varphi(p))+a_2(p)E_2(\varphi(p))+a_3(p)E_3(\varphi(p)).
\end{equation}
\end{theorem}
\begin{proof}
First, assume the existence of an immersion $\varphi$ satisfying (\ref{ecuacion2}). If we consider the vector field $N(p)=N_1(p)E_1(\varphi(p))+N_2(p)E_2(\varphi(p))+N_3(p)E_3(\varphi(p))$, then (\ref{e11}) is equivalent to (\ref{e6}), that is,
$$
\varphi_z=\frac{i}{\sqrt{K}}\ N\times_{\varphi}\nabla_{\varphi_z}N.
$$
Hence, $N$ is a unit normal vector field and so $N^e$ is its Gauss map. Moreover, $K$ must be its extrinsic curvature, and the conformal structure of $\Sigma$ is that of its second fundamental form because
$$
\langle\varphi_z,-\nabla_{\varphi_z}N\rangle=\frac{i}{\sqrt{K}}\langle N\times_{\varphi}\nabla_{\varphi_z}N,-\nabla_{\varphi_z}N\rangle=0.
$$

Finally, we want to show existence and uniqueness to the first order PDE 
\begin{equation}\label{ecuacion}
\varphi(p)^{-1}\varphi_z(p)=a_1(p)E_1+a_2(p)E_2+a_3(p)E_3,
\end{equation}
in the simply connected surface $\Sigma$. To do this, we apply
the classical Frobenius Theorem, and so we need to prove that $[\varphi_z,\varphi_{\oz}]=0$. But, as we have shown previously, this condition is equivalent to (\ref{e8bis}). 
Hence, once we prescribe an initial condition $\varphi(p_0)=q_0\in M$, there exists a unique immersion $\varphi:\Sigma\fl M$ such that (\ref{ecuacion}) (or (\ref{ecuacion2})) is satisfied with $\varphi(p_0)=q_0$.
\end{proof}
\begin{remark}
From the comments previous to the above theorem, it should be observed that if there is no solution $(a_1,a_2,a_3)$ to (\ref{e11}), so that (\ref{e8bis}) is satisfied, then there is no immersion such that $N^e$ is its Gauss map, $K$ is its extrinsic curvature and the conformal structure of $\Sigma$ is that of its second fundamental form.
\end{remark}

\subsection{Unimodular Lie groups.}\label{unimodular}

Since the Lie bracket and the cross product in the Lie algebra $\mathfrak{m}$ are skew-symmetric bilinear forms, they are related by a unique endomorphism $L: \mathfrak{m}\longrightarrow \mathfrak{m}$ satisfying $L(X\times Y)=[X,Y]$, for all $X,Y\in \mathfrak{m}.$

The Lie group $M$ is called \emph{unimodular} if $L$ is self-adjoint \cite[Lemma 4.1]{JMilnor}.

From now on, let us assume that $M$ is unimodular. Then, since $L$ is self-adjoint, there exists a positively oriented orthornormal basis $\{ E_1, E_2, E_3\}$ of $\mathfrak{m}$ of eigenvalues of $L$, that is
\begin{equation}\label{e1}
[E_2,E_3]=c_1 E_1, \quad [E_3,E_1]=c_2 E_2, \quad [E_1,E_2]=c_3 E_3,
\end{equation}
for constants $c_1,c_2,c_3 \in\mathbb{R}$. Throughout this section, we will work with the left invariant frame $\{E_1(q), E_2(q), E_3(q)\}$ of $M$,  such that $E_i(e)=E_i,$  $i=1,2,3 $ satisfy \eqref{e1}, where $e$ is the identity element of $M$.

As usual, we define
\[
\mu_1=\dfrac{-c_1+c_2+c_3}{2}, \quad\mu_2=\dfrac{c_1-c_2+c_3}{2}, \quad\mu_3=\dfrac{c_1+c_2-c_3}{2}.
\]
In terms of these new constants we can write the connection of $M$ as

\begin{equation}\label{simbolos}
\begin{array}{lll}
\nabla_{E_1(q)}E_1(q)=0  \quad  & \nabla_{E_1(q)}E_2(q)=\mu_1 E_3(q) \quad & \nabla_{E_1(q)}E_3(q)=-\mu_1 E_2(q)\\[15pt]
\nabla_{E_2(q)}E_1(q)=-\mu_2 E_3(q) \quad & \nabla_{E_2(q)}E_2(q)=0 \quad & \nabla_{E_2(q)}E_3(q)=\mu_2 E_1(q)\\[15pt]
\nabla_{E_3(q)}E_1(q)=\mu_3 E_2(q)  \quad  & \nabla_{E_3(q)}E_2(q)=-\mu_3 E_1(q) \quad & \nabla_{E_3(q)}E_3(q)=0.
\end{array}
\end{equation}

In \cite[Section 2.6]{Perez} the authors classified, in terms of $\mu_i,\  i=1,2,3$, the unimodular Lie groups. Here we present some of them:
\begin{itemize}
\item $\real^3$, for $\mu_1=\mu_2=\mu_3=0$.
\item $\s^3$, for $\mu_1=\mu_2=\mu_3=1$.
\item The Berger spheres, for $\mu_1=\mu_2= \tau, \mu_3= \dfrac{1- 2\tau^2}{2\tau}$, with $\tau>0$.
\item $\widetilde{PSL}_2(\real, \tau)$, for $\mu_1=\mu_2=-\tau , \mu_3= \dfrac{1+ 2\tau^2}{2\tau}$, with $\tau>0$.
\item $\mathrm{Nil}_3(\tau)$, for $\mu_1=\mu_2=\tau , \mu_3= -\tau$, with $\tau>0$.
\item $\mathrm{Sol}_3$, for $\mu_1= -1, \mu_2=1 ,\mu_3= 0$.
\end{itemize}
These spaces are simply connected homogeneous 3-dimensional manifolds. With exception of the Berger spheres, these spaces belong to the eight Thurton's geometries. For further details about them see \cite{Benoit}.

As a consequence, we obtain  the following result for unimodular Lie groups.

\begin{proposition}\label{p2}
Let $\varphi:\Sigma\fl M$ be an isometric immersion on a unimodular Lie group $M$ and  $\{E_1,E_2,E_3\}$ be a positively oriented orthonormal basis of $\mathfrak{m}$ satisfying \eqref{e1}. Assume that $\varphi$ has positive extrinsic curvature $K$ and Gauss map
$g: \Sigma \fl\mathbb{C}\cup\{\infty\}$. Let $z$ be a conformal parameter for the second fundamental form and denote
$$
\varphi^{-1}\varphi_z=a_1 E_1+ a_2 E_2+a_3E_3\equiv (a_1,a_2,a_3).
$$
Then,
\begin{eqnarray}
&&\quad a_1=\dfrac{\overline{g}_z (A B\sqrt{
     K} - i ((A+2) (B-2) \mu_2 +
       2 g C \mu_3)) -
 g_z (  \overline{A} B\sqrt{K} - i ( (\overline{A}+2)( B-2) \mu_2 +
       2 \overline{g}C  \mu_3))}{B (B^2 K - \widetilde{C}^2 \mu_2 \mu_3 + \mu_1 ( C^2 \mu_3-(B-2)^2 \mu_2) - i\sqrt{K} (\vert A\vert^2 \mu_1 + \vert A+2\vert^2 \mu_2 +  4 \vert g\vert^2 \mu_3))}
\nonumber\\[2mm]
&&\quad a_2=\dfrac{ \overline{g}_z (A (B-2) \mu_1-iB\sqrt{K} (A+2)   - 2 g \widetilde{C} \mu_3) -
 g_z (\overline{A} (B-2)  \mu_1 +i  B  \sqrt{K}(\overline{A}+2) + 2  \overline{g} \widetilde{C} \mu_3)}{B (B^2 K -  \widetilde{C}^2 \mu_2 \mu_3 + \mu_1 (C^2 \mu_3-(B-2)^2 \mu_2 ) - i\sqrt{K} (\vert A\vert^2  \mu_1 + \vert A+2\vert^2 \mu_2 + 4 \vert g \vert^2 \mu_3))}
\label{unim}\\[2mm]
&&\quad a_3=\dfrac{i\left(
    (\mu_1-\mu_2)(g(g_z-g^2\overline{g}_z)- \overline{g}(\overline{g}_z- \overline{g}^2 g_z))+ LB(g \overline{g}_z- \overline{g} g_z)\right)}{B (B^2 K -  \widetilde{C}^2 \mu_2 \mu_3 + \mu_1 (C^2 \mu_3-(B-2)^2 \mu_2 ) - i\sqrt{K} (\vert A\vert^2  \mu_1 + \vert A+2\vert^2 \mu_2 + 4 \vert g \vert^2 \mu_3))},\nonumber
\end{eqnarray}
where,
$
A=g^2-1,\quad
B=1+\vert g\vert^2,\quad
C=g-\overline{g},\quad
\widetilde{C}=g+\overline{g},\quad
L= 2 i \,\sqrt{K} +\mu_1+\mu_2.
$
\end{proposition}
\begin{proof}
A straightforward computation gives us that the discriminant of the system of linear equations (\ref{e11}) is
$$
\dfrac {B^2  K - \widetilde {C}^2 \mu_2 \mu_3 + \mu_1 (-(B - 2)^2 \mu_2 + C^2 \mu_3) -  i \sqrt {K} (\vert A \vert^2 \mu_1 + \vert A + 2\vert^2 \mu_2 +
       4 \vert g \vert^2\mu_ 3)} {B^2 K},
$$
where $N^e$ is written in terms of $g$, and the Christoffel symbols $\Gamma_{ij}^k$ are obtained from (\ref{simbolos}).

Since this function does not vanish anywhere we have a unique solution in the unknowns $a_1,a_2,a_3$. Thus, a long, but straightforward, computation gives (\ref{unim}).
\end{proof}

From the previous expression of the functions $a_i$ we can compute the first and second fundamental form of the immersion $\varphi:\Sigma\fl M$. Bearing in mind that the expressions are too long, we will write this calculation for the specially interesting  case of unimodular Lie groups with an isometry group of dimension larger than or equal to four. In this case we can assume $\mu_1=\mu_2$.

\begin{corollary}\label{c3}
Let $\varphi:\Sigma\fl M$ be an isometric immersion in a unimodular Lie group $M$ such that $\mu_1=\mu_2$. Assume that $\Sigma$ has positive extrinsic curvature $K$ and Gauss map
$g: \Sigma \fl \mathbb{C}\cup\{\infty\}$. Let $z$ be a conformal parameter for the second fundamental form. Then,  the coefficients of the first fundamental form are  written in terms of $g$ and $K$ as
\begin{eqnarray*}
&&E=\dfrac{(\overline{g}_z R_1 +
   g_z (2 i B^2 \sqrt{K} + R_2)) ( g_z \overline{R}_1+ \overline{g}_z (2 i B^2 \sqrt{K} +R_2)
  )}{(B (B^2 K -4\vert g \vert^2\mu_1 \mu_3 -(B-2)^2 \mu_1^2 - i\sqrt{K} R_2))^2}
\\
&&F=\dfrac{\overline{g}_{\overline{z}} (\overline{g}_z R_1 R_2+
    g_z R_3)
       +
 g_{\overline{z}} (g_z  R_2 \overline{R}_1
      +
   \overline{g}_z R_3))}{\vert B (B^2 K -4\vert g \vert^2\mu_1 \mu_3 -(B-2)^2 \mu_1^2 - i\sqrt{K}R_2)\vert^2}.
\end{eqnarray*}

Moreover, the second fundamental form $II=\sqrt{ - D \ K}\ |dz|^2$, where $D:= \vert E\vert^2- F^2$  is given by
\[
D=-\dfrac{4 B^4\left( R_4(\vert g_z\vert^2- \vert \overline{g}_z\vert^2)+
i \sqrt{K}(g_z g_{\overline{z}} \, \overline{R}_1 - \overline{g}_z \overline{g}_{\overline{z}} \, R_1)\right)^2}{\vert B (B^2 K -4\vert g \vert^2\mu_1 \mu_3 -(B-2)^2 \mu_1^2 - i\sqrt{K}R_2)\vert^4}.
\]
Here, $A, B, C$ and $\widetilde{C}$ were defined in Proposition \ref{p2} and
\begin{eqnarray*}
R_1&=&A^2 \mu_1 - (A+2)^2 \mu_1 + 4 g^2 \mu_3\\
 R_2&=&\vert A \vert^2 \mu_1 + \vert A+2\vert^2 \mu_1 +
      4\vert g\vert^2 \mu_3\\
      R_3&=&2 B^4 K + \vert A\vert^4\mu_1^2 + \vert A+2\vert^4\mu_1^2 +8\vert  g\vert^2 (B-2)^2 \mu_1 \mu_3 + 16\vert g\vert^4  \mu_3^2
       -2  C^2 \widetilde{C}^2 \mu_1^2\\
       R_4&=& B^2 K + 4\vert g \vert^2 \mu_1 \mu_3 + (B-2)^2 \mu_1^2.
\end{eqnarray*}
\end{corollary}
Once $a_1,a_2,a_3$ are computed in terms of the Gauss map and the extrinsic curvature, observe that $N^e$ and $K$ must satisfy the equation (\ref{e8bis}). For instance, in the case of $\real^3$ this fact determines $K$ in terms of $N^e$, up to a positive constant, (see \cite[Theorem 3]{GM}). Let us study this equation for unimodular Lie groups with an isometry group of dimension larger than or equal to four.
\begin{theorem}
\label{t1}
Let $\varphi:\Sigma\fl M$ be an isometric immersion in a unimodular Lie group $M$. Assume that $\varphi$ has positive extrinsic curvature $K$ and Gauss map
$g: \Sigma \fl \mathbb{C}\cup\{\infty\}$. Let $z$ be a conformal parameter for the second fundamental form. Then, when $\mu_1=\mu_2$, the Gauss map satisfies
\begin{equation}
\label{e13}
g_{z\overline{z}}= G_1 g_z g_{\overline{z}}+ G_2 g_z \overline{g}_{\overline{z}}+ G_3 \overline{g}_z g_{\overline{z}}+ G_4 \overline{g}_z\overline{g}_{\overline{z}}+ G_5 K_z g_{\overline{z}}+ G_6 K_z \overline{g}_{\overline{z}}+ G_7 K_{\overline{z}} g_z + G_8 K_{\overline{z}} \overline{g}_z,
\end{equation}
here $G_i:= G_i(K)$ where, for $P_{(m_0,m_1,m_2, m_3,m_4)}= \displaystyle\sum_{i=0}^{i=4} m_i\vert g\vert^{2i}$,
\begin{eqnarray*}
G_1(x)&=&\dfrac{2 \overline{g} (B^6 x^2 +
   2 B^2 x (P_{(0,4,2,1,1)} \mu_1^2 + P_{(1,-5,1,3,0)} \mu_1 \mu_3 +
      P_{(0,5,3,0,0)} \mu_3^2) -32 \vert g \vert^4 (B-2) \mu_1\mu_3^3 )}{B^3 (B^4 x^2 + \mu_1^2 ((B-2)^2 \mu_1 +
      4 \vert g\vert^2 \mu_3)^2 +
   2 x (P_{(1,0,6,0,1)} \mu_1^2 +
      P_{(0,4,-8,4,0)} \mu_1 \mu_3 + 8\vert g\vert^4 \mu_3^2))}\\[15pt]
& &- \dfrac{2\overline{g}
       \mu_1 ((1
        +6 \vert g \vert^2 -
            \vert g \vert^4 P_{(13,-10,3,0,1)}) \mu_1^3 +
      2 (B-2) P_{(-1,-6,18,2,3)} \mu_1^2 \mu_3 +
      P_{(0,2,66,-18,14)} \mu_1 \mu_3^2)}{B^3 (B^4 x^2 + \mu_1^2 ((B-2)^2 \mu_1 +
      4 \vert g\vert^2 \mu_3)^2 +
   2 x (P_{(1,0,6,0,1)} \mu_1^2 +
      P_{(0,4,-8,4,0)} \mu_1 \mu_3 + 8\vert g\vert^4 \mu_3^2))}
\end{eqnarray*}
\begin{eqnarray*}
G_2(x)&=&\dfrac{2 i g (B-2) (\mu_3 - \mu_1) (2 B^4 x^2 + B^4 x \mu_1^2
-iB^2\sqrt{x}(x+ \mu_1^2)P_{(\mu_1, 2\mu_3,\mu_1,0,0)}
 - \mu_1^2 P_{(\mu_1,4\mu_3-2\mu_1, \mu_1,0,0)}^2)}{B^3 \sqrt{x} (x + \mu_1^2) \left(B^4 x +  P_{(\mu_1,4\mu_3-2\mu_1, \mu_1,0,0)}^2\right)}
\end{eqnarray*}
\begin{eqnarray*}
G_3(x)&=&\dfrac{2 i g (B-2) (\mu_1 - \mu_3) (2 B^4 x^2 + B^4 x \mu_1^2 +
i \sqrt{x} B^2(x+\mu_1^2)P_{(\mu_1, 2\mu_3,\mu_1,0,0)}
 - \mu_1^2 P_{(\mu_1,4\mu_3-2\mu_1, \mu_1,0,0)}^2)}{B^3 \sqrt{x} (x + \mu_1^2) \left(B^4 x +  P_{(\mu_1,4\mu_3-2\mu_1, \mu_1,0,0)}^2\right)}
\end{eqnarray*}
\begin{eqnarray*}
G_4(x)&=&-\dfrac{4 g^3 (B-2) (\mu_1 - \mu_3)^2 \left(-3 B^2 x + P_{(\mu_1^2, 8\mu_1\mu_3-6\mu_1^2, \mu_1^2,0,0)}\right)}{B^3 (x + \mu_1^2) \left(B^4 x +  P_{(\mu_1,4\mu_3-2\mu_1, \mu_1,0,0)}^2\right)}
\\[10pt]
G_5(x)&=&\overline{G}_7(x)=\dfrac{B^2 x + \mu_1 P_{(\mu_1,4\mu_3-2\mu_1, \mu_1,0,0)}}{4 x (\sqrt{x} + i\mu_1) \left(B^2 \sqrt{x} + i P_{(\mu_1,4\mu_3-2\mu_1, \mu_1,0,0)}\right)}
\\[10pt]
G_6(x)&=&\overline{G}_8(x)=-\dfrac{i g^2 (\mu_1 - \mu_3)}{\sqrt{x} (\sqrt{x} + i\mu_1) \left(B^2 \sqrt{x} + i P_{(\mu_1,4\mu_3-2\mu_1, \mu_1,0,0)}\right)}\ .
\end{eqnarray*}

Moreover, the integrability equation \eqref{e8} is equivalent to the equation \eqref{e13}.
\end{theorem}

\begin{proof}
First, once we have the immersion from $\Sigma$ to $M$ the vector field   $U=(u_1, u_2, u_3)=i \varphi^{-1} [\varphi_z, \varphi_{\overline{z}}]$  in the Lie algebra $\mathfrak{m}$ of $M$ is a null vector field, in particular, satisfies trivially  the complex equation $u_1+i u_2+g u_3=0$. Then, using this equation and its conjugate we have a system in the unknowns $\{g_{z\overline{z}}, \overline{g}_{z\overline{z}}\}$,  whose discriminant 
\[-\dfrac{16 B^2 K}{ (B^4 K^2 + \mu_1^2 ((B-2)^2 \mu_1 +   4 \vert g\vert^2 \mu_3)^2 + 2 K (P_{(1,0,6,0,1)} \mu_1^2 + 
 P_{(0,4,-8,4,0)} \mu_1 \mu_3 + 8\vert g\vert^4 \mu_3^2))}\]
is different from zero.  After a tedious computation, we can solve the system obtaining equation \eqref{e13}.  So we proved that if $\varphi^{-1} [\varphi_z, \varphi_{\overline{z}}]=(0,0,0)$, then $g$ satisfies equation \eqref{e13}.

Reciprocally, observe that it is easy to check that a real vector $w\equiv(w_1,w_2,w_3)\in \mathfrak{m}$ satisfies that
\begin{equation}\label{e10}
w\text{ is parallel to }N^e\text{ if, and only if, }(|g|^2-1)(w_1+i\,w_2)-2g\,w_3=0.
\end{equation}

Hence, if we consider $w\equiv\mathfrak{Re}(V)$ in (\ref{e8bis}) then we obtain a unique complex equation in the unknowns $\{g_{z\overline{z}}, \overline{g}_{z \overline{z}}\}$. 
A direct computation shows that if $g$ satisfies \eqref{e13}, then $g$ solves equation \eqref{e10}.     
\end{proof}

From Theorem \ref{teo1}, we note that Proposition \ref{p2} and Theorem \ref{t1} also give the sufficient conditions for a map $N^e$ and a positive function $K$ to be the Gauss map and extrinsic curvature of an immersion $\varphi:\Sigma\fl M$, for the simply connected Riemann surface $\Sigma$.

Though the previous equations look like difficult to handle, they take an easier form when we fix $M$. Thus, we will particularize for the 3-sphere and obtain some consequences.

\subsection{The canonical Sphere}\label{CS}
In this section, we focus our attention on the canonical sphere $\mathbb{S}^3$. From the classification of unimodular Lie groups, $\mu_1=\mu_2=\mu_3=1$.

We consider a left invariant orthonormal vector field  $E_i(q),$ $i=1,2,3,$ in $\s^3$ such that $E_i(e)=E_i$, where $E_i$ satisfies \eqref{e1} and  $e$ is the identity element of $\mathbb{S}^3$.

Let $\psi:\Sigma\fl \s^3$ be an isometric immersion in $\s^3$. As before, we denote by
$$
a_i(p)=\langle \psi_z(p), E_i(\psi(p))\rangle,\qquad i=1,2,3.
$$

Then, as a consequence of Proposition \ref{p2} and Corollary \ref{c3}, we have
\begin{corollary}\label{c1}
Let $\psi:\Sigma\fl \s^3$ be an isometric immersion in $\s^3$ with positive extrinsic curvature $K$ and Gauss map $g:\Sigma\fl \mathbb{C}\cup\{\infty\}$. Let $z$ be a conformal parameter for the second fundamental form, then we have
\begin{equation}
\label{e19}
\begin{array}{ccl}
a_1&=&\dfrac{(1- \overline{g}^2 )g_z + ( g^2-1) \overline{g}_z}{(1 + \vert g\vert^2)^2 ( \sqrt{K}-i)} \\[15pt]
a_2&=&-\dfrac{i ((1 + \overline{g}^2) g_z + (1 + g^2) \overline{g}_z)}{(1 + \vert g\vert^2)^2 ( \sqrt{K}-i))} \\[15pt]
a_3&=&\dfrac{2 (\overline{g} g_z - g \overline{g}_z)}{(1 + \vert g\vert^2)^2 ( \sqrt{K}-i))}
\end{array}
\end{equation}

Moreover, the first and second fundamental forms of the immersions are given by
\begin{eqnarray*}
&I=-\dfrac{4 g_z \overline{g}_z}{(1 + \vert g \vert^2)^2 (-i + \sqrt{K})^2} \ dz^2 + \dfrac{2 (\vert g_z\vert^2 +\vert \overline{g}_z\vert^2)}{(1 + \vert g\vert^2 )^2 (1 + K)} \ \vert dz\vert^2-\dfrac{4 g_{\overline{z}} \overline{g}_{\overline{z}}}{(1 + \vert g \vert^2)^2 (i + \sqrt{K})^2}\  d\overline{z}^2,  \\
&II= 2 \,\dfrac{\sqrt{K} \ (\vert g_{\overline{z}}\vert^2- \vert  g_z\vert^2 )}{(1 + \vert g\vert^2 )^2 (1 + K)}\ \vert dz\vert^2.
\end{eqnarray*}
\end{corollary}

On the other hand, the integrability condition $[\psi_z,\psi_{\overline{z}}]=0$ gives us how the extrinsic curvature can be recovered in terms of the Gauss map,
\begin{proposition}
\label{p3}
Let $\psi:\Sigma\fl\s^3$ be an isometric immersion in  $\s^3$. Assume that $\Sigma$ has positive extrinsic curvature $K$ and Gauss map
$g: \Sigma \fl\mathbb{C}\cup\{\infty\}$.  Let $z$ be a conformal parameter for the second fundamental form. Then, $ \vert g_{\overline{z}}\vert^2-\vert g_z\vert^2>0$ and the equation $[\psi_z,\psi_{\overline{z}}]=0$ is equivalent to
\begin{equation}
\label{e16}
 \left( \log\left(\dfrac{1-K-2 i \sqrt{K}}{\sqrt{K}}\right)\right)_z  =\dfrac{2}{ \vert g_{\overline{z}}\vert^2-\vert g_z\vert^2}\left(\overline{g}_z g_{z\overline{z}}- g_z \overline{g}_{z\overline{z}}+2g_z \overline{g}_z\, \dfrac{ g \overline{g}_{\overline{z}}-\overline{g} g_{\overline{z}}  }{1 + \vert g \vert^2}\right).
\end{equation}
\end{proposition}
\begin{proof}
Since the second fundamental form of the immersion is positive definite, we have from Corollary \ref{c1} that $ \vert g_{\overline{z}}\vert^2-\vert g_z\vert^2$ must be positive.

On the other hand,
$$
\nabla_{\psi_z} \psi_{\overline{z}}=\sum_{i}\overline{a_i}_z\,E_i\circ\varphi+\sum_{i,j,k}a_i\overline{a_j}\Gamma_{ij}^k\, E_k\circ\varphi.
$$
Moreover, from \eqref{simbolos}, the Christoffel symbol $\Gamma_{ij}^k=\pm 1$ when $\{i,j,k\}=\{1,2,3\}$ and vanishes otherwise. So,
$$
0=\psi^{-1}\ [\psi_z,\psi_{\oz}]\equiv(\overline{a_1}_z-{a_1}_{\oz}+2a_2\overline{a_3}-2a_3\overline{a_2},\
\overline{a_2}_z-{a_2}_{\oz}+2a_3\overline{a_1}-2a_1\overline{a_3},\
\overline{a_3}_z-{a_3}_{\oz}+2a_1\overline{a_2}-2a_2\overline{a_1}).
$$

If we denote by $V=(V_1, V_2, V_3)$ the previous vector and use \eqref{e19}, then the equality  $V_1+iV_2+g V_3=0$ is equivalent to
\begin{equation}
\label{e24}
 d_1\,\dfrac{(-i + \sqrt{K})}{K(i+ \sqrt{K})}  K_z+ d_2\,\dfrac{(i+\sqrt{K})}{K(-i+\sqrt{K})} K_{\overline{z}}+ d_3=0,
\end{equation}
  where
  \begin{eqnarray*}
  d_1 &=&  (1 + \vert g\vert^2) g_{\overline{z}}\\
d_2 &=&(1 + \vert g\vert^2) g_{z}  \\
d_3 &=& 8 \overline{g}\, g_{\overline{z}}\, g_z - 4 g_{z\overline{z}}(1+\vert g\vert^2)
\end{eqnarray*}

Thus, taking equation \eqref{e24} and its conjugated, we obtain a system of linear equations in the unknowns $\dfrac{(-i + \sqrt{K})}{K(i+ \sqrt{K})}  K_z$,
$\dfrac{(i+\sqrt{K})}{K(-i+\sqrt{K})} K_{\overline{z}}$. Bearing in mind that
\begin{equation}
\label{e27}
 2\left( \log\left(\dfrac{1-K-2 i \sqrt{K}}{\sqrt{K}}\right)\right)_z=\dfrac{(-i + \sqrt{K})}{K(i+ \sqrt{K})}  K_z,
\end{equation}
and that the discriminant of the system is $ (1 + \vert g \vert^2)^3 (\vert g_z\vert^2 - \vert g_{\overline{z}} \vert^2)$, which is different from zero, we obtain (\ref{e16}).

Finally, it is easy to show that if \eqref{e16} is satisfied then the vector $V$ vanishes identically, that is, $[\psi_z,\psi_{\overline{z}}]=0$.
\end{proof}

As an interesting consequence we have:
\begin{corollary} 
Let $\psi:\Sigma\fl \s^3$ be an isometric immersion in  $\s^3$, with positive extrinsic curvature $K$, and Gauss map
$g: \Sigma\fl\mathbb{C}\cup\{\infty\}$. Then,  $K$ is a positive constant if, and only if, $g$ is a harmonic map for the conformal structure induced by the second fundamental form.
\end{corollary}
\begin{proof}
From (\ref{e16}) and (\ref{e27}), $K$ is constant if, and only if, $\overline{g}_z g_{z\overline{z}}- g_z \overline{g}_{z\overline{z}}+2g_z \overline{g}_z\, \dfrac{ g \overline{g}_{\overline{z}}-\overline{g} g_{\overline{z}}  }{1 + \vert g \vert^2}=0$. Thus, considering this equation and its conjugate in the unknowns $g_{z\oz}$ and $\overline{g}_{z\oz}$, $K$ to be constant  is equivalent to
$$
g_{z\oz}-2g_z g_{\oz} \frac{\overline{g}}{1+|g|^2}=0,
$$
as we wanted to show.
\end{proof}

From Theorem \ref{teo1} and Proposition \ref{p3}, we obtain necessary and sufficient conditions for the existence of an immersion with positive extrinsic curvature in $\s^3$ in terms of its Gauss map and the conformal structure of the second fundamental form. More concretely:
\begin{corollary}\label{t2}
Let $\Sigma$ be a simply connected Riemann surface and  $g:\Sigma \fl \mathbb{C}\cup\{\infty\}$ a differentiable map.
Then, there exists an isometric immersion $\psi:\Sigma\fl\s^3$ with Gauss map $g$ and such that the conformal structure induced by the second fundamental form is that of $\Sigma$ if, and only if,
\begin{equation}
\label{e17}
\vert g_{\overline{z}}\vert^2-\vert g_z\vert^2>0
\end{equation}
and there exists a positive function $K$ solving the equation \eqref{e16}.

Moreover, $\psi$ is unique up to left translations.
\end{corollary}

It is well known that there exists a correspondence between isometric simply connected surfaces having constant mean curvature $H_1$ in $\mathbb{R}^3$ and constant mean curvature $H_2$ in $\s^3$, with $H_1^2=H_2^2+1$. On the other hand given a surface $\Sigma$ having positive constant mean curvature in $\mathbb{R}^3$ or constant mean curvature in $\s^3$, there exists a surface $\Sigma^\prime$  parallel to $\Sigma$ having positive constant extrinsic curvature (probably, with singularities). As a consequence, one can find a correspondence between simply connected surfaces having positive constant extrinsic curvature in $\mathbb{R}^3$ and simply connected surfaces having constant extrinsic curvature in $\s^3$. The following result is a generalization of this correspondence.

\begin{proposition}\label{pp1}
Let $\psi:\Sigma\fl\s^3$ be an immersion having positive extrinsic  curvature $K$ and Gauss map $g$. Assume
\begin{equation}
\label{e20} \left(\dfrac{K_z}{2 \sqrt{K}(1+ K)}\right)_{\overline{z}}=0.
\end{equation}
Then, there exists an immersion $\psi^{\ast}:\Sigma\fl\real^3$ having the same Gauss map $g$, the same conformal structure for the second fundamental form, and positive extrinsic curvature $K^{\ast}$, where $K^{\ast}$ is a solution of
\begin{equation}
\label{e21}
(\log K^{\ast})_{z}=\left( \log\left( \dfrac{(i+ \sqrt{K})^4}{K}\right)\right)_z.
\end{equation}

Conversely, let $\psi^{\ast}:\Sigma\fl\real^3$ be an immersion having Gauss map $g$ and positive extrinsic curvature $K^{\ast}$.  If there exists a positive function $K:\Sigma\fl \real$, such that \eqref{e21} is satisfied, then, there is an isometric immersion $\psi:\Sigma\fl\s^3$ having positive extrinsic curvature $K$, the same Gauss map $g$, and the same conformal structure for the second fundamental form.
\end{proposition}
\begin{proof}
In \cite[Theorem 3]{GM}, the authors gave a necessary and sufficient condition for the existence of an immersion from $\Sigma$ in $\mathbb{R}^3$ having a given Gauss map and conformal structure for the second fundamental form. So, if we want to obtain an immersion $\psi^{\ast}:\Sigma\fl\real^3$ having the same Gauss map $g$ and the same conformal structure for $II$, then $g$ must satisfy
\begin{equation}
\label{e22}
\vert g_{\overline{z}}\vert^2- \vert g_{z}\vert^2>0;
\end{equation}
and
\begin{equation}
\label{e23}
\mathfrak{Im}\left( \left( \dfrac{4}{\vert g_{\overline{z}}\vert^2- \vert g_{z}\vert^2}\left(\overline{g}_z g_{z\overline{z}}- g_z \overline{g}_{z\overline{z}}+2g_z \overline{g}_z\, \dfrac{ g \overline{g}_{\overline{z}}-\overline{g} g_{\overline{z}}  }{1 + \vert g \vert^2}\right) \right)_{\overline{z}}\right)=0,
\end{equation}
where $\mathfrak{Im}(\cdot)$ denotes the imaginary part of a complex number.

Note that from Proposition \ref{p3} the equation \eqref{e22} is satisfied. Using equation \eqref{e16}, we show that \eqref{e23} is equivalent to \eqref{e20}. In fact, by \eqref{e16} and \eqref{e27}, for $R=\sqrt{K}$, we have
\[
\mathfrak{Im}\left( \left(\dfrac{R-i}{R (i + R)} \ 2  R_z\right)_{\overline{z}}\right)=\mathfrak{Im}\left( \left( \dfrac{4}{\vert g_{\overline{z}}\vert^2- \vert g_{z}\vert^2}\left(\overline{g}_z g_{z\overline{z}}- g_z \overline{g}_{z\overline{z}}+2g_z \overline{g}_z\, \dfrac{ g \overline{g}_{\overline{z}}-\overline{g} g_{\overline{z}}  }{1 + \vert g \vert^2}\right) \right)_{\overline{z}}\right)
\]
\[
\mathfrak{Im}\left( \left(\dfrac{R-i}{R (i + R)} \ 2  R_z\right)_{\overline{z}}\right)=0\ \Longleftrightarrow\
(2 R R_{\overline{z}}R_z)-(1+R^2)R_{z\overline{z}}=0\  \Longleftrightarrow \ \left( \dfrac{R_z}{1+R^2}\right)_{\overline{z}}=0.
\]
The last equality  is equivalent to equation \eqref{e20}.

The converse is a direct consequence of Corollary \ref{t2}.
\end{proof}


\section{Surfaces with negative extrinsic curvature}
This section is devoted to surfaces having negative extrinsic curvature. Let $\Sigma$ be an orientable smooth surface and $\varphi : \Sigma\fl M$ be an isometric immersion in an oriented 3-dimensional Riemannian manifold $M$. Assume that $\varphi$ has negative extrinsic curvature $K$, that is, such that the determinant of its shape operator is negative. Then, the second fundamental form is a Lorentzian metric and $\Sigma$ can be considered as a Lorentz surface. So, there exist local proper null coordinates $(u,v)$, and we can choose $N$, a global unit normal vector field to $\Sigma$ such that the first and second fundamental forms of the immersion can be written as
$$
I:=\langle d \varphi, d \varphi\rangle = \mathbf{E} du^2+ 2\mathbf{F} dudv+\mathbf{G} dv^2,
$$
$$
II:= \langle - d N, d\varphi \rangle=\  2\mathbf{f}dudv,
$$
where $\mathbf{f}$ is a positive function. We will say that the unit vector field $N$ is compatible with such coordinates if $\mathbf{f}$ is a positive function. \\
The extrinsic curvature of the surface $\Sigma$ is given by $K=-\dfrac{\mathbf{f}^2}{\mathbf{D}}$, where $\mathbf{D}= \mathbf{E}\mathbf{G}-\mathbf{F}^2>0$. And, proceeding as in the previous case, we have
\begin{lemma}
Let $\varphi:\Sigma\ \fl\ M$ be an isometric immersion from an orientable  Riemannian surface $\Sigma$ into a 3-dimensional Riemannian manifold $M$. Assume that $\varphi$ has negative  extrinsic curvature $K$. Then, the Weingarten equations are
\begin{equation}\label{n1}
\nabla_{\varphi_u}N=\dfrac{\mathbf{f}}{\mathbf{D}}(\mathbf{F}\varphi_u- \mathbf{E}\varphi_v)
\quad \textnormal{and} \quad
\nabla_{\varphi_v}N=\dfrac{\mathbf{f}}{\mathbf{D}}(-\mathbf{G}\varphi_u+ \mathbf{F}\varphi_v),
\end{equation}
where  $N$ is a unit normal vector field compatible with  local proper null coordinates $(u,v)$ for the second fundamental form. Moreover, 
\begin{equation}\label{n2}
\varphi_u=\dfrac{1}{\sqrt{-K}} N\times_{\varphi}\nabla_{\varphi_u}N\quad\textnormal{and} \quad \varphi_v=-\dfrac{1}{\sqrt{-K}} N\times_{\varphi}\nabla_{\varphi_v}N,
\end{equation}
and the integrability equation $[\varphi_u,\varphi_v]=0$, is equivalent to
\begin{equation}\label{n3}
N\times_{\varphi}\left(\nabla_{\partial_u}\left(\frac{1}{\sqrt{-K}}\nabla_{\partial_{v}} N\right)+\nabla_{\partial_{v}}\left(\frac{1}{\sqrt{-K}}\nabla_{\partial_{u}} N\right)\right)=0.
\end{equation}
\end{lemma}

In order to recover the immersion $\varphi$ in terms of its Gauss map and its negative extrinsic curvature, we assume that $M$ is a Riemannian Lie group with a left invariant  vector field $\{E_1(q), E_2(q), E_3(q)\}$, $q\in\ M$. Again, for simplicity, we set  $E_i:=E_i(e), \ i=1,2,3$, where $e$ is the identity element of $M$.

Let $(u,v)$  be a proper null coordinates for the second fundamental form, and $N$ be a unit normal vector field compatible with such coordinates. Considering the Lie group structure, we write 
\[
\varphi(p)^{-1}\varphi_u(p)= \sum_{i=1}^{3}\langle\varphi_u(p), E_i(\varphi(p)) \rangle \, E_i=:\mathbf{a}_1 E_1+\mathbf{a}_2 E_2+\mathbf{a}_3 E_3,
\]
\[
\varphi(p)^{-1}\varphi_v(p)= \sum_{i=1}^{3}\langle\varphi_v(p), E_i(\varphi(p)) \rangle \, E_i=:\mathbf{A}_1 E_1+\mathbf{A}_2 E_2+\mathbf{A}_3 E_3,
\]
\[
N^e(p)\equiv (N_1, N_2, N_3)= \sum_{i=1}^{3}\langle N(p), E_i(\varphi(p)) \rangle \, E_i.
\]
So, we can rewrite (\ref{n2}) as
\begin{equation}
\label{n22}
(\mathbf{a}_1, \mathbf{a}_2, \mathbf{a}_3)=\dfrac{1}{\sqrt{-K}} (N_1, N_2, N_3) \times \left((N_1)_u+\sum_{i,j} N_i \mathbf{a}_j \Gamma_{ij}^1 , (N_2)_u+\sum_{i,j} N_i \mathbf{a}_j \Gamma_{ij}^2 ,  (N_3)_u+\sum_{i,j} N_i \mathbf{a}_j \Gamma_{ij}^3 \right)
\end{equation}
\begin{equation*}
(\mathbf{A}_1, \mathbf{A}_2, \mathbf{A}_3)=\dfrac{-1}{\sqrt{-K}} (N_1, N_2, N_3) \times \left((N_1)_v+\sum_{i,j} N_i \mathbf{A}_j \Gamma_{ij}^1\, , \,(N_2)_v+\sum_{i,j} N_i \mathbf{A}_j \Gamma_{ij}^2\, , \, (N_3)_v+\sum_{i,j} N_i \mathbf{A}_j \Gamma_{ij}^3 \right),
\end{equation*}
where $\Gamma_{ij}^k$ are the Christoffel symbols associated with the orthonormal basis $\{E_1,E_2,E_3\}$. Here, $\times$ denotes the standard cross product in the Lie algebra $\mathfrak{m}$.

The equations in (\ref{n22}) can be considered as a system of linear equations in the unknowns $a_1,a_2,a_3$. Hence, as long as the discriminant of the above system is different from zero, we can determine $\varphi(p)^{-1}\,\varphi_u(p)$ and $\varphi(p)^{-1}\,\varphi_v(p)$ in terms of the Gauss map and the extrinsic curvature of the immersion.  As in the case when the extrinsic curvature is positive, the immersion $\varphi$ can be computed, up to left translations,  if we have  $\varphi(p)^{-1}\,\varphi_u(p)$ and $\varphi(p)^{-1}\,\varphi_v(p)$. Thus, we will say that the immersion is determined in terms of its Gauss map and its extrinsic curvature.

The unit normal vector  $N$ and the extrinsic curvature  $K$ must satisfy the integrability condition (\ref{n3})  which can  be rewritten in terms of $N^e$ and $K$, using the left translations in the Riemannian Lie group $M$. That is, if we denote by $V$ the left translation of $\nabla_{\partial_u}\left(\frac{1}{\sqrt{-K}}\nabla_{\partial_{v}} N\right)+\nabla_{\partial_{v}}\left(\frac{1}{\sqrt{-K}}\nabla_{\partial_{u}} N\right)$ to the identity element $e$, the equation (\ref{n3}) is equivalent to
\begin{eqnarray}\label{n33}
&N^e \text{ is parallel to  }V. 
\end{eqnarray}
Then, the Gauss map and extrinsic curvature of the immersion $\varphi$ must satisfy (\ref{n22}) and (\ref{n33}). 

Proceeding as in the case of positive extrinsic curvature, we obtain the following results which we will omit their proofs since they are analogous.

\begin{theorem}\label{teo11}
Let $M$ be a Riemannian Lie group, and $\{E_1(q),E_2(q),E_3(q)\}$ a left invariant orthonormal frame with associated Christoffel symbols $\Gamma_{ij}^k$. Consider a simply connected Lorentz surface $\Sigma$,   $N^e:\Sigma\fl\s^2\subseteq\mathfrak{m} $ a differential map and $K:\Sigma\fl\R$ a negative function. Assume that $(\mathbf{a}_1,\mathbf{a}_2,\mathbf{a}_3)$ and $(\mathbf{A}_1,\mathbf{A}_2,\mathbf{A}_3)$ are smooth solutions to (\ref{n22}), so that (\ref{n33}) is satisfied.

Then, there exists a unique immersion $\varphi:\Sigma\fl M$, up to left translations, such that $N^e$ is its Gauss map, $K$ is its extrinsic curvature and the structure of $\Sigma$ is that of the second fundamental form induced by $\varphi$, with
\begin{eqnarray}\label{ecuacion22}
\varphi_u(p) & = & \mathbf{a}_1(p)E_1(\varphi(p))+\mathbf{a}_2(p)E_2(\varphi(p))+ \mathbf{a}_3(p)E_3(\varphi(p)), \\
\nonumber \varphi_v(p) & = & \mathbf{A}_1(p)E_1(\varphi(p))+\mathbf{A}_2(p)E_2(\varphi(p))+ \mathbf{A}_3(p)E_3(\varphi(p)).
\end{eqnarray}

\end{theorem}

Now, we focus our attention on unimodular 3-dimensional Riemannian Lie groups $M$ with isometry group of dimension larger than or equal to 4. 

\begin{proposition}\label{p4}
Let $\varphi:\Sigma\ \fl\ M$ be an oriented isometric immersion on a unimodular Lie group $M$ with $\mu_1=\mu_2$, and $\{E_1,E_2,E_3\}$ be a positively oriented orthonormal basis satisfying \eqref{e1}. Assume that $\varphi$ has negative extrinsic curvature $K$,  with
$$K \neq -\mu_1^2, \quad K\neq -\dfrac{((-1 +\vert g\vert^2)^2 \mu_1 + 4\vert g\vert^2 \mu_3)^2}{(1+\vert g\vert^2)^4},$$
where $(u,v)$ are proper null coordinates  for the second fundamental form and $g$ a Gauss map of $\Sigma$  compatible with such coordinates.  Denote 
\begin{eqnarray*}
\varphi^{-1}\varphi_u & = & \mathbf{a}_1 E_1+\mathbf{a}_2 E_2+\mathbf{a}_3 E_3, \\
\varphi^{-1}\varphi_v  & = & \mathbf{A}_1 E_1+\mathbf{A}_2 E_2+\mathbf{A}_3 E_3.
\end{eqnarray*}
Then,
\vspace{.5cm}
\newline
$
\mathbf{a}_1=\dfrac{i (\sqrt{-K}
       B (\overline {A} g_u -A \overline {g} _u) + (B -
        2) (-(\overline {A} + 2) g_u + (A + 2) \overline {g} _u) \mu_
      1 + 2 C (\overline {g} g_u + g \overline {g} _u) \mu_
      3)} {B (\sqrt{-K} - \mu_ 1) (\sqrt{-K}  B^2 - (B - 2)^2\mu_ 1 -
     4 \vert g \vert^2 \mu_ 3)}\\
$
\vspace{.3cm}
\newline
$
\mathbf{a}_2=\dfrac{-\sqrt {-K} B (g_u (\overline {A} + 2 ) + \overline {g} _u (A + 2)) +
 (B - 2) (\overline {A} g_u + A \overline {g} _u)\mu_ 1 +
 2 \widetilde {C}  (\overline {g} g_u + g \overline {g} _u) \mu_ 3} {B (\sqrt{-K} - \mu_ 1) (\sqrt{-K}  B^2 - (B - 2)^2\mu_ 1 -
     4 \vert g \vert^2 \mu_ 3)}
$
\vspace{.3cm}
\newline
$
\mathbf{a}_3=\dfrac{2 i  (-\overline{g} g_u + g \overline{g}_u)}{ \sqrt{-K}  B^2 - (B - 2)^2\mu_ 1 -  4 \vert g \vert^2 \mu_ 3},
$
\vspace{.3cm}
\newline
$
\mathbf{A}_1=-\dfrac{i (\sqrt{-K}
       B (\overline {A} g_v -A \overline {g} _v) + (B -
        2) ((\overline {A} + 2) g_v - (A + 2) \overline {g} _v) \mu_
      1 - 2 C (\overline {g} g_v + g \overline {g} _v) \mu_
      3)} {B (\sqrt{-K} + \mu_ 1) (\sqrt{-K}  B^2 + (B - 2)^2\mu_ 1 +
     4 \vert g \vert^2 \mu_ 3)}\\
$
\vspace{.3cm}
\newline
$
\mathbf{A}_2=\dfrac{\sqrt {-K} B (g_v (\overline {A} + 2 ) + \overline {g} _v (A + 2)) +
 (B - 2) (\overline {A} g_v + A \overline {g} _v)\mu_ 1 +
 2 \widetilde {C}  (\overline {g} g_v + g \overline {g} _v) \mu_ 3} {B (\sqrt{-K} + \mu_ 1) (\sqrt{-K}  B^2 + (B - 2)^2\mu_ 1 + 4 \vert g \vert^2 \mu_ 3)}
$
\vspace{.3cm}
\newline
$
\mathbf{A}_3=-\dfrac{2 i  (-\overline{g} g_v + g \overline{g}_v)}{ \sqrt{-K}  B^2 + (B - 2)^2\mu_ 1 +  4 \vert g \vert^2 \mu_ 3},
$
\vspace{.3cm}\newline
where
$
A=g^2-1,\quad
B=1+\vert g\vert^2,\quad
C=g-\overline{g},\quad
\widetilde{C}=g+\overline{g}.
$
\end{proposition}

Now, let us present a partial differential equation involving the Gauss map and the negative extrinsic curvature for the case when the ambient Riemannian manifold $M$ is a unimodular Lie groups with an isometry group of dimension larger than or equal to four.

\begin{theorem}
\label{t3}
Let $\varphi:\Sigma\ \fl\ M$ be an oriented isometric immersion in a unimodular Lie group $M$. Assume that $\Sigma$ has negative extrinsic curvature $K$ with 
$$K \neq -\mu_1^2, \quad K\neq -\dfrac{((-1 +\vert g\vert^2)^2 \mu_1 + 4\vert g\vert^2 \mu_3)^2}{(1+\vert g\vert^2)^4},$$ 
where $(u,v)$  is a proper null coordinates  for the second fundamental form and  
$g: \Sigma \ \fl\ \mathbb{C}\cup\{\infty\}$ a Gauss map of $\Sigma$  compatible with such coordinates. Then, when $\mu_1=\mu_2$, the Gauss map satisfies
\begin{equation}
\label{e25}
g_{uv}= \mathbf{G}_1 g_u g_v+ \mathbf{G}_2 g_u \overline{g}_v+ \mathbf{G}_3 \overline{g}_u g_v+ \mathbf{G}_4 \overline{g}_u\overline{g}_v-\mathbf{G}_5 K_u g_v- \mathbf{G}_6 K_u \overline{g}_v-\mathbf{G}_7 K_v g_u - \mathbf{G}_8 K_v \overline{g}_u,
\end{equation}
here $\mathbf{G}_i:= G_i(K)$, where $G_i(x)$ is defined in Theorem \ref{t1}.
\end{theorem}

Again we will particularize for the case when $M$ is the 3-sphere $\s^3$.

\subsection{Surfaces with negative extrinsic curvature in $\s^3$} In this section we work with surfaces immersed in $\s^3$, for this, let us fix a left invariant orthonormal frame $\{E_1(p), E_2(p), E_3(p)\}$ on $\s^3$, as before, $E_i$ stands for $E_i(e)$, where $e$ is the identity element of $\s^3$, and $\{E_1, E_2,E_3\}$ satisfies \eqref{e1}.

Let $\psi:\Sigma\ \fl\ \s^3$ be an orientable isometric immersion in  $\s^3$ having negative extrinsic curvature $K,\ K\neq -1$. Let $(u,v)$ be a local proper null coordinates for the second fundamental form and  
$g: \Sigma \ \fl\ \mathbb{C}\cup\{\infty\}$ a Gauss map of $\Sigma$  compatible with such coordinates. Then, the first fundamental form of $\Sigma$ is given by
\[
I_{\s^3}=\dfrac{4 g_u \overline{g}_u}{(1 -\sqrt{-K})^2 (1 + \vert g\vert^2)^2} du^2+\dfrac{4(\overline{g}_u g_v + g_u \overline{g}_v)}{(1 +K) (1 +\vert g\vert^2)^2}
du dv +\dfrac{4 g_v \overline{g}_v}{(1 +\sqrt{-K})^2 (1 + \vert g\vert^2)^2}dv^2.\]

In the next proposition,  the integrability condition $[\psi_u,\psi_v]=0$  gives a relationship between the  negative extrinsic curvature and the Gauss map.
\begin{proposition}
\label{c2}
Let $\psi:\Sigma\ \fl\ \s^3$ be an orientable isometric immersion in  $\s^3$. Assume that $\Sigma$ has negative extrinsic curvature $K,\ K\neq -1$. Let $(u,v)$ be a proper null coordinates for the second fundamental form and  
$g: \Sigma \ \fl\ \mathbb{C}\cup\{\infty\}$ a Gauss map of $\Sigma$  compatible with such coordinates. Then, $i(  g_u \overline{g}_v-\overline{g}_u g_v)>0$, and the integrability equation $[\psi_u,\psi_v]=0$ is equivalent to
\begin{equation}
\label{e26}
\left\{
\begin{array}{ccc}
 \left( \log\left(\dfrac{(1+ \sqrt{-K})^4}{-K}\right)\right)_u &=&\dfrac{4}{ \overline{g}_u g_v - g_u \overline{g}_v}\left(\overline{g}_u g_{uv}- g_u \overline{g}_{uv}+2\vert g_u\vert^2\, \dfrac{ g \overline{g}_v-\overline{g} g_v  }{1 + \vert g \vert^2}\right)\\[15pt]
\left( \log\left(\dfrac{(-1+ \sqrt{-K})^4}{-K}\right)\right)_v&=&-\dfrac{4}{ \overline{g}_u g_v - g_u \overline{g}_v}\left(\overline{g}_v g_{uv}- g_v \overline{g}_{uv}+2\vert g_v \vert^2\, \dfrac{ g \overline{g}_u-\overline{g} g_u  }{1 + \vert g \vert^2}\right)
\end{array}
\right. .
\end{equation}
\end{proposition}
As a consequence, we obtain  a  corollary which also was proved in \cite[Theorem 3.1]{BIK}.
\begin{corollary}
Let $\psi:\Sigma\ \fl\ \s^3$ be an oriented isometric immersion in  $\s^3$. Assume that $\Sigma$ has negative extrinsic curvature $K,\ K\neq -1$.  Let $(u,v)$ be a local proper null coordinates for the second fundamental form and  
$g: \Sigma \ \fl\ \mathbb{C}\cup\{\infty\}$ a Gauss map of $\Sigma$  compatible with such coordinates. Then, $K$ is a negative constant if, and only if, $g$ is Lorentz harmonic map for these proper null coordinates.

\end{corollary}

From Theorem \ref{teo11} and Proposition \ref{c2}, we obtain necessary and sufficient conditions for the existence of an immersion with negative extrinsic curvature in $\s^3$ in terms of its Gauss map and the proper null coordinates determined by the second fundamental form. More precisely:

\begin{corollary}
Let $\Sigma$ be a simply connected Lorentz surface and $g:\Sigma\ \fl \ \mathbb{C}\cup\{\infty\}$ a differential map. Then, there exists an isometric immersion $\psi:\Sigma\ \fl\ \s^3$ having Gauss map $g$,  negative extrinsic curvature $K$ and such that the Lorentz structure on $\Sigma$ is that one induced by the second fundamental form  if, and only if,
\[
i(  g_u \overline{g}_v-\overline{g}_u g_v)>0,
\]
and there exists a negative function $K:\Sigma\ \fl\ \ \real-\{-1\}$ \  \  solving the equation \eqref{e26}.

Moreover, $\psi$ is unique up to left translations.
\end{corollary}

Finally, we present a correspondence between simply connected surfaces having negative extrinsic curvature in $\R^3$ and simply connected surfaces having negative extrinsic curvature in $\s^3$.

\begin{proposition}\label{p5}
Let $\Sigma$ be a simply connected Riemannian surface and  $\psi:\Sigma\ \fl \ \s^3$ be an oriented isometric immersion having negative extrinsic  curvature $K, \ K\neq -1, $. Let $(u,v)$ be a proper null coordinates for the second fundamental form of $\psi$ and  
$g: \Sigma \ \fl\ \mathbb{C}\cup\{\infty\}$ a Gauss map of $\Sigma$  compatible with such coordinates.  Assume
\begin{equation}
\label{e30}  \left(\dfrac{-4 K_u}{ \sqrt{-K}(1+K)}\right)_{v}=0.
\end{equation}

Then, there exists an isometric immersion $\psi^*:\Sigma\ \fl \ \real^3$ having the same Gauss map $g$. The parameters  $(u,v)$ are proper null coordinates for the second fundamental form of $\psi^*$, $g$ is compatible with such coordinates. Moreover, the immersion has  negative extrinsic curvature $K^*$, where $K^*$ is a solution of
\begin{equation}
\label{e31}
\left\{
\begin{array}{cc}
  \left( \log\left(\dfrac{(1+ \sqrt{-K})^4}{-K}\right)\right)_u &= \left(\log\left( K^*\right)\right)_u\\[15pt]
\left( \log\left(\dfrac{(-1+ \sqrt{-K})^4}{-K}\right)\right)_v&=\left(\log\left( K^*\right)\right)_v
\end{array}
\right. .
\end{equation}

Conversely, let $\psi^*:\Sigma\ \fl \ \real^3$ be an oriented isometric immersion from a simply connected Riemannian surface, having  negative extrinsic curvature $K^*$. Let $(u,v)$ be a proper null corrdinates  and $g$ a Gauss map compatible with such coordinates. If there exists a negative function $K:\Sigma\ \fl\  \real-\{-1\}$, such that \eqref{e31} is satisfied, then, there is an isometric immersion $\psi:\Sigma\ \fl \ \s^3$ having negative extrinsic curvature $K$ and the same Gauss map $g$. Moreover, the parameters $(u,v)$ are proper null coordinates  for the second fundamental form of $\psi$ and the Gauss map $g$ is compatible with such coordinates.

\end{proposition}

\section{Example}
In this section, we construct an example of a surface $\Sigma$ immersed isometrically in the canonical sphere $\s^3$, having constant negative extrinsic curvature $K=-2$. In order to do that, we identify $\s^3$ with $\mathrm{SU}(2)$,  the group of $2\times2$ unitary matrix with determinant 1. The identification is given by
\[
(z,w)\in \ \s^3\subset\, \real^4=\mathbb{C}^2\ \longleftrightarrow \ 
\left(\begin{array}{cc}
 z& w\\
 -\overline{w}& \overline{z}
\end{array}
\right)\ \in\ \mathrm{SU}(2).
\]
A basis of the Lie algebra  $\mathfrak{su}(2)$ of $\s^3$ satisfying \eqref{e1} is 
\[
E_1=\left(\begin{array}{cc}
 0& 1\\
 -1& 0
\end{array}
\right), \  \  \ \ 
E_2=\left(\begin{array}{cc}
 0& i\\
 i & 0
\end{array}
\right), \ \ \  \
E_3=\left(\begin{array}{cc}
 i& 0\\
 0& -i
\end{array}
\right).
\]

We use Proposition \ref{p5}  to construct a surface in $\s^3$ having negative constant extrinsic curvature.  So, we take the pseudo sphere in $\real^3$ parametrized by
\[
\psi^{\real} = \left( \sech \left(\dfrac{v - u}{2}\right) \cos\left(\dfrac{v +u}{2}\right), 
  \sech\left(\dfrac{v - u}{2}\right)  \sin\left(\dfrac{v + u}{2}\right), \dfrac{v - u}{2}  - \tanh\left(\dfrac{v - u}{2}\right) \right).
\]
The pseudo sphere in $\real^3$ has extrinsic curvature $K^*=-1$, the parameters $(u,v)$ are proper null coordinates for the second fundamental form. The unit normal vector compatible with this proper null coordinates is  

$N=(N_1,N_2, N_3)=$
\[\sqrt{e^{u+v}(e^u-e^v)^2}\left( -\dfrac{e^{-(u+v)}}{2} \cos( \frac{u+v}{2}) \sech( \frac{v-u}{2}),  -\dfrac{e^{-(u+v)}}{2} \sin( \frac{u+v}{2}) \sech( \frac{v-u}{2}), \dfrac{2}{e^{2u}-e^{2v}}\right).\]

The Gauss map of the pseudo sphere is 
\[
g=\dfrac{N_1+ i N_2}{1-N_3}=\dfrac{e^{\frac{(-1+i)(u+v)}{2}} (e^u-e^v)\sqrt{e^{u+v}(e^u-e^v)^2}}{e^{2v}-e^{2u}+2\sqrt{e^{u+v}(e^u-e^v)^2}}.
\]

 Note that  any negative constant $K$ different from $-1$ satisfies the equation \eqref{e31}, since $K^*=-1$.  In particular, Proposition \ref{p5} ensures the existence of an immersion $\psi$ in $\s^3$ having extrinsic curvature $K=-2$ and the same Gauss map  $g$ of the pseudo sphere.  From Proposition \ref{p4} the vector fields
\begin{eqnarray*}
 \psi^{-1}\psi_u  & = &  \mathbf{a}_1 E_1+\mathbf{a}_2 E_2+\mathbf{a}_3 E_3,\\
 \psi^{-1}\psi_u   & = & \mathbf{A}_1 E_1+\mathbf{A}_2 E_2+\mathbf{A}_3 E_3,
\end{eqnarray*} 
are  determined by  
\begin{eqnarray*}
\mathbf{a}_1&=&  \dfrac{(1+\sqrt{2})e^{\frac{u+v}{2}}((e^u-e^v)\cos(\frac{u+v}{2})+(e^u+e^v)\sin(\frac{u+v}{2}))}{(e^u+e^v)^2}   \\[15pt]
\mathbf{a}_2&=&  -\dfrac{(1+\sqrt{2})e^{\frac{u+v}{2}}((e^u+e^v)\cos(\frac{u+v}{2})-(e^u-e^v)\sin(\frac{u+v}{2}))}{(e^u+e^v)^2}    \\[15pt]
\mathbf{a}_3&=&  \dfrac{(1+\sqrt{2})}{2} \tanh^2(\frac{u-v}{2}) \\[15pt]
\mathbf{A}_1&=&  \dfrac{(-1+\sqrt{2})e^{\frac{u+v}{2}}((-e^u+e^v)\cos(\frac{u+v}{2})+(e^u+e^v)\sin(\frac{u+v}{2}))}{(e^u+e^v)^2}    \\[15pt]
\mathbf{A}_2&=&-\dfrac{(-1+\sqrt{2})e^{\frac{u+v}{2}}((e^u+e^v)\cos(\frac{u+v}{2})+(e^u-e^v)\sin(\frac{u+v}{2}))}{(e^u+e^v)^2}      \\[15pt]
\mathbf{A}_3&=&-  \dfrac{(-1+\sqrt{2})}{2} \tanh^2(\frac{u-v}{2}). 
\end{eqnarray*}
We consider the change of variables  $2t = v+u, \ 2s= v-u$. Then 
\begin{eqnarray*}
 \psi^{-1}\psi_s  & = &  \mathbf{b}_1 E_1+\mathbf{b}_2 E_2+\mathbf{b}_3 E_3,\\
 \psi^{-1}\psi_t   & = & \mathbf{B}_1 E_1+\mathbf{B}_2 E_2+\mathbf{B}_3 E_3,
\end{eqnarray*}
where $\mathbf{b}_i=\mathbf{A}_i-\mathbf{a}_i, \ \mathbf{B}_i=\mathbf{A}_i+\mathbf{a}_i, \ i=1,2,3$. Note that  the immersion $\psi$ is a revolution surface and since  $\mathbf{b}_3$ does not depend on $t$, we can assume that
 $$\psi(s,t)=(\sin(\alpha(s)) \sin(\beta(s)) \cos(t), \sin(\alpha(s)) \sin(\beta(s))\sin(t), \cos(\alpha(s)) \sin(\beta(s)), \cos(\beta(s))),$$
for some real functions $\alpha,\, \beta$.  Recall that, the vector field $\psi^{-1}\psi_s$ in the Lie algebra of $\s^3$ is the left translation of $\psi_s$ to the identity of $\s^3$.  As we identified $\s^3$ with $\mathrm{SU}(2)$, the left translation is the product of matrices. So, the functions $\alpha(s), \beta(s)$ solve the system $\psi_s(s,0)=\psi(s,0) \cdot\psi_s(s,0)$, with initial condition $\alpha(0)=\dfrac{\pi}{2}, \beta(0)=\dfrac{\pi}{2}$. We use the  stereographic projection in order to plot such a surface, see figure bellow.
\begin{figure}[h]
\begin{center}
\fbox{\includegraphics[scale=.46,trim = 0mm 0mm 0mm 0mm,clip]{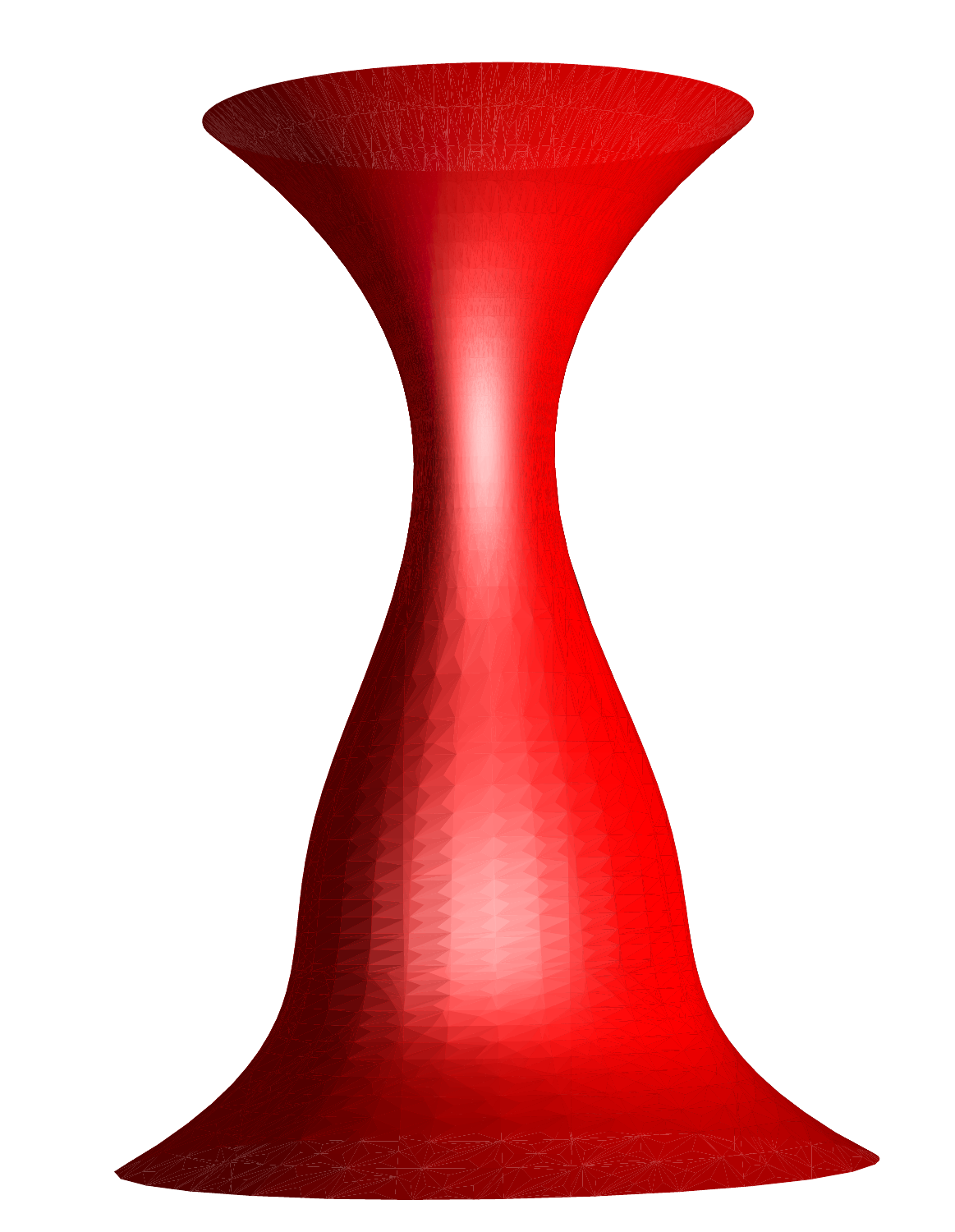}}\hfil
\fbox{\includegraphics[scale=1.0, trim = 5mm 46mm 7mm 36mm, clip]{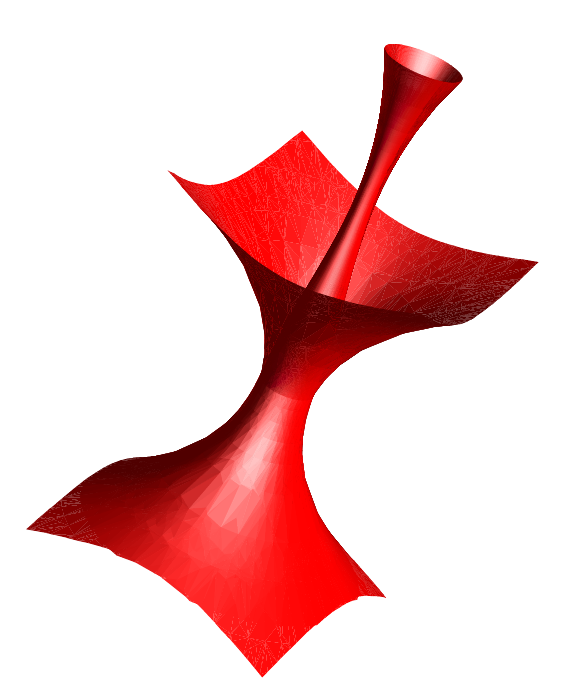}}
\end{center}
\end{figure}

\vspace{1cm}

\begin{tabular}{l|l}
Abigail Folha  &  Carlos Pe\~{n}afiel\\
abigailfolha@vm.uff.br & penafiel@im.ufrj.br\\
Universidade Federal Fluminense\ &  Universidade Federal de Rio de Janeiro\\
Instituto de Matem\'{a}tica- Departamento de Geometria & Instituto de Matem\'{a}tica e Estat\'{i}stica\\
R. M\'{a}rio Santos Braga, s/n   &  Av. Athos da Silveira Ramos 149\\
Campus do Valonguinho & CT, bl. C, Cidade Universit\'{a}ria\\
CEP 24020-140  &  CEP 21941-909\\
Niter\'{o}i, RJ - Brasil. & Rio de Janeiro, RJ - Brasil.\\

\end{tabular}

\end{document}